\documentclass[12pt]{amsart}
\usepackage{epsf} 
\usepackage{amssymb,latexsym, amsmath, amscd, array}
\swapnumbers 
\numberwithin{equation}{section}
\def\m{\medskip}

\theoremstyle{plain} 
\newtheorem{thm}{Theorem}[section]
\newtheorem{cory}[thm]{Corollary} 
\newtheorem{cor}[thm]{Corollary} 
\newtheorem{lem}[thm]{Lemma}  
\newtheorem{lemma}[thm]{Lemma} 
\newtheorem{theorem}[thm]{Theorem}

\newtheorem{prop}[thm]{Proposition} 

\newcommand\theoref{Theorem~\ref} 
\newcommand\lemref{Lemma~\ref} 

\newcommand\propref{Proposition~\ref} 
\newcommand\corref{Corollary~\ref}

\theoremstyle{definition} 
\newtheorem{defin}[thm]{Definition}

\newtheorem{rem}[thm]{Remark}

\def\eps{\varepsilon} 
\def\gf{\varphi}

\def\Z{{\mathbb Z}}

 \def\ar{\mathfrak a}
 \def\ber{\mathfrak b}
\def\cer{\mathfrak c}

\def\Hom{\operatorname{Hom}}
\def\Tor{\operatorname{Tor}}
\def\TC{\operatorname{TC}}

\def\cat{\operatorname{cat}}

\def\zcl{\operatorname{zcl}}
\def\zp{\Z[\pi]}
\def\cd{\operatorname{cd}}

\def\ts{\times}

\def\ot{\otimes}
\def\ov{\overline}

\def\wt{\widetilde}

\long\def\forget#1\forgotten{} %

\begin{document}

\title[On spaces of minimal higher topological complexity.]{On spaces of minimal higher topological complexity.}

\author[Yu.~Rudyak]{Yuli B. Rudyak}

\address{Yuli B. Rudyak \newline
Department of Mathematics, University of Florida \newline
358 Little Hall, Gainesville, FL 32611-8105, USA}
\email{rudyak@ufl.edu}

\vspace{3mm}

\begin{abstract}
Let $\TC_n(X)$ denote the $n$-th topological complexity of a topological space X. 
It is known that $\TC_n(X)$ does not exceed $n-1$ for non-contractible $X$, and so it makes sense to describe spaces $X$ with $\TC_n(X) =n-1$. Grant--Lupton--Oprea proved the following: If X is a nilpotent space with $\TC_n(X)=n-1$ then $X$ is homotopy equivalent to an odd-dimensional sphere. Here we made an attempt to get rid of nilpotency condition and prove the following: If $\TC_n(X) =n-1$ then either $X$ is homotopy equivalent to a sphere of odd dimension or is a homology circle with the infinite cyclic fundamental group.
\end{abstract}

\maketitle


\vspace{3mm}

\section{Introduction}
The $n$th topological complexity $\TC_n(X)$ of a topological space $X$ is the sectional category of a fibrational substitute of the diagonal map $\Delta:X\to X^n$. This concept appeared in \cite{R}, see \cite{BGRT} or \cite[Section 2]{GLO} for greater details. Note that the inequality $\TC_n(X)\geq n-1$ holds for all (non-contractible) $X$ and $n$, and so it seems reasonable to describe the spaces $X$ with $\TC_n(X)=n-1$. It is easy to see $\TC_n(S^{2k+1})=n-1$ for $n\geq 2$ and $k\geq 0$,~\cite{BGRT}. Thus it is quite natural to conjecture that the converse is also true: If $\TC_n(X)=n-1$ then $X$ is homotopy equivalent to an odd-dimensional sphere. Currently we have two theorems below.

\begin{thm}\label{t:glo}\cite[Theorem 3.4]{GLO}
Let X be a connected CW space of finite type. If $X$ is simply connected then $X$ is homotopy equivalent to some sphere $S^{2r+1}$ of odd dimension, $r\geq 1$. If $X$ is not simply connected then $X$ is either acyclic space or homology sphere. Additionally, if $X$ is a nilpotent space or a co-H-space then $X$ is a homotopy circle.
\end{thm}

\m Note that the proof in \cite[Lemma 3.3]{GLO} has a small error that is corrected in \cite{A}. However, the statement of \theoref{t:glo} is correct.

\m
We have the following slight improvement of \theoref{t:glo}.

\begin{thm}[\theoref{main}]
Let X be a connected CW space of finite type.
If $\TC_n(X) =n-1, n\geq 2$, then $X$ is either homotopy equivalent to some sphere $S^{2r+1}$ of odd dimension, $r\geq 0$ or a homology circle with $\pi_1(X)=\Z$.
\end{thm}

We emphasize again: It is an open question whether the space $X$ with $\TC_n(X)=n-1, n\geq 3$ and $\pi_1(X)=\Z$ is a homotopy circle.

\m We note the following known fact.

\begin{prop}\label{p:mintc}
If the CW  space $X$ is not contractible then $\TC_n(X)\geq n-1$.  
\end{prop}

\begin{proof}
We know that $\TC_n(X)\geq \cat(X^{n-1})$, \cite[\S 3]{BGRT}. Furthermore, if $X$ is not acyclic then cup-length of $X$ is at least 1, and so cup-length of $X^{n-1}$ is at least $n-1$, and hence $\cat (X^{n-1})\geq n-1$ by \cite[Proposition 1.5]{CLOT}. 
If $X$ is acyclic but non-contractible then $\pi_1(X)$ contains a non-trivial cyclic subgroup $G$. Take a covering map $\wt X$ of $X$ with $\pi_1(\wt X)=G$. Now, 
\[
\TC_n(X)\geq \cat (X^{n-1})\geq \cat(\wt {X^{n-1}})=\cat((\wt X)^{n-1})\geq n-1
\] 
because $\cat (Y)\geq\cat (\wt Y)$ for all $Y$ and all its coverings. To prove this known fact, note the following: If $p: \wt Y   \to Y$ is a covering map and a subspace $U\subset Y$ is contractible in $Y$ then $p^{-1}(U)$ is contractible in $\wt Y$, because of the covering homotopy property.
\end{proof}

\m All spaces are assumed to be connected CW spaces of finite type.

\m We use the sign $\simeq$ for  isomorphism of groups and homotopy equivalences of spaces.

\section{Zero Divisors Cup Length}
\begin{defin} Fix a natural number $n\geq 2$. Let $X$ be a connected CW space and $d = d_n : X \to X^n$ the diagonal
map, $d(x) = (x, \cdots, x)$. The {\em $n$-fold zero divisor ideal for $X$} is the kernel of the homomorphism 
\[
d^* : H^*(X^n) \to H^*(X)
\]
induced by the diagonal.
A non-zero element $x\in X^n$ is called a {\em $n$-fold zero divisor element} if $d^*(x) = 0$.  
A {\em $n$-fold zero divisor cup length} of $X$ denoted by
$\zcl_n(X)$ is a maximal number $k$ such that there are classes 
\[
u_i\in H^*(X^n;A_i), i=1, \ldots, k,\, \dim u_i>0
\]
 where the $A_i$'s are the local coefficient systems, with the following property: 
\[
\begin{aligned}
d^* u_i&=0 \text{ for all } i=1, \ldots, k, \text{ and } \\
0\neq &u_1\smile\cdots \smile u_k\in H^*(X^n;A_1\otimes\cdots \otimes A_k).
\end{aligned}
\]
\end{defin}

\begin{thm}\label{t:zcn}
$\TC_n(X)\geq\zcl_n(X)$ for all local coefficient system.
\end{thm}

\begin{proof} 
This is a special case of \cite[Theorem 4]{Sv}.
\end{proof}

\section{Berstein--Schwarz Class}
Consider a pointed connected $CW$ space $(X, x_0)$. Put $\pi=\pi_1(X,x_0)$,  $\zp$ the group ring of $\pi$, and
$I(\pi)$ the augmentation ideal of $\zp$,  i.e. the kernel of the augmentation map $\eps:
\zp \to \Z$. 

\m The cohomological dimension of the group $\pi$ is denoted by $\cd (\pi)$.

\m
Let $p:\wt X\to X$ be the universal covering map for $X$ and $\wt X_0=p^{-1}(x_0)$. The exactness of the sequence
\[
0\to H_1(\wt X, \wt X_0) \to H_0(\wt X_0) \to H_0(\wt X)
\]
yields an isomorphism $H_1(\wt X, \wt X_0)\cong I(\pi)$.  So, we get
isomorphisms
\[
H^1(X,x_0;I(\pi))\cong \Hom_{\zp}(H_1(\wt X,\wt
X_0),I(\pi))\cong\Hom_{\zp}(I(\pi),I(\pi)),
\]
see \cite[Text before Theorem 2.51]{CLOT} or \cite[(3.7)]{Ber}.

\m We denote by $\ov \ber\in H^1(X,x_0;I(\pi))$ the element that
corresponds to the identity $1_{I(\pi)}$ in the right-hand side of the last isomorphism.  (Here we use local coefficient system for the group $I(\pi)$.)
Finally, we set
\begin{equation}\label{eq:bersteinclass}
\ber=\ber_X=j^*\ov \ber\in H^1(X;I(\pi))
\end{equation}
where $j:X=(X,\emptyset)\to (X,x_0)$ is the inclusion of pairs. We call $\ber_X$ the {\em Berstein-Schwarz class} of $X$. 

\begin{rem}
To justify the name "Berstein--Schwarz class", note that the class appeared in~\cite[p.99]{Sv} (implicitly) and~\cite[(3.7)]{Ber}.
\end{rem}

\m Let $B\pi$ denote the classifying space for $\pi$ and put $\ber_{\pi}:=\ber_{B\pi}$. Consider its cup power 
\begin{equation}\label{e:smile}
\ber_{\pi}^m=\ber_{\pi}\smile \cdots \smile \ber_{\pi}\in H^m(B{\pi};I(\pi)^{\otimes m})
\end{equation}
 where $I(\pi)^{\otimes m}=I(\pi){\otimes} \cdots \otimes I(\pi)$ ($m$ factors) and $\otimes$ denotes the tensor product over $\Z$.

\begin{theorem}\label{t:power}
If $\cd (\pi) =m\leq \infty $ then $\ber^n_{\pi}\neq 0$ for all $n\leq m$.
\end{theorem}

\begin{proof}
See e.g. \cite[Theorem 3.4]{DR}. Cf. \cite[Prop. 34]{Sv}.
\end{proof}

\section{K\"unneth Theorem in the Present of Local Coefficients}

Let me cite the following fact, \cite[Theorem 1.7]{G}.

\begin{thm}\label{t:kunneth}
Let $K_1$ and $K_2$ be path connected CW complexes of finite type and $K_3 = K_1 \times K_2$. Let $G_i$ be left modules over $\pi_1(K_i),i=1,2$ where
$G_3 = G_1 \otimes G_2$ and the action of $\pi_1(K_3)$ on $G_3$ is the "product action". If $G_1$ or $G_2$
is a free abelian group, then there exists a natural exact sequence
\begin{equation}
\CD
\begin{aligned}
0&@>>> H^*(K_1;G_1)\otimes H^*(K_2;G_2) @>\mu>> H^*(K_3;G_3)\\
&\ @>>> \Tor[H^*(K_1;G_1), H^*(K_2;G_2)] @>>> 0
\end{aligned}
\endCD
\end{equation}
where all cohomology is with local coefficients and the maps are of degree 0 and 1.
\end{thm}

\begin{cor}\label{cc:kunneth}
Let $X$ be a path connected CW space of finite type, $\pi=\pi_1(X)$, and $I(\pi)$ the augmentation ideal of $\pi$. Then we have a monomorphism \label{mu}
\[
\mu\colon (H^*(X;I(\pi))^{\otimes n}\to H^*(X^n; (I(\pi)^{\otimes n}).
\]
\end{cor}

\begin{proof}
Note that $I(\pi)$ is a free abelian group. Now, the corollary follows from \theoref{t:kunneth}.
\end{proof}

\m Note also the following fact. Given two spaces $X,Y$ with $\pi_1(X)=\pi=\pi_1(Y)$, consider the projections $p:X\ts Y \to X$ and $q: X\ts Y \to Y$, and   $x\in H^*(X;A$ and $y\in H^*(Y;B)$.

\begin{prop}\label{p:smile}
$\mu(x\otimes y)=p^*(x)\smile q^*(y)\in H^*(X\ts Y; A\otimes B)$.
\end{prop}

\begin{proof}
Use \cite[VII.7.6 and VII.8.15]{D}.
\end{proof}

\section{Proof of Theorem 1.2}

Given a connected CW space $X$ of finite type and a particular natural $n$ such that $\TC_n(X)=n-1$. The goal is to prove that $X$ is either homotopy equivalent to some sphere $S^{2r+1}$ of odd dimension, $r\geq 0$ or a homology circle with $\pi_1(X)=\Z$. Since the simply connected case is treated in \cite{GLO} we concentrate the exposition for non-simply connected case. As before, we put $\pi=\pi_1(X)$ and denote  by $I(\pi)$ the augmentation for $\pi$.

\m Recall the Berstein-Schwarz class $\ber=\ber_{X}\in H^1(X;I(\pi))$ in \eqref{eq:bersteinclass}.

\m Let $p_i: X^n\to X$ be the projection on $i$th factor. Consider the induced map 
\[
p^*_i: H^*(X;- )\to H^*(X^n;-).
\]
Define
\[
\ber_i:  =p_i^*(\ber)\in H^*(X^n;I(\pi)).
\]

\forget
\begin{lem}\label{l:proj}
$\cer_i=\mu(\ar_i)$ for $\mu$ in \eqref{mu}.
\end{lem}

\begin{proof} Because of \propref{p:smile} we have
\[
\begin{aligned} 
\mu(\ar _i)&=\mu(1\otimes 1\otimes\cdots\otimes 1 \otimes \ber \otimes 1\otimes \cdots \otimes -\ber) \\
&=p_1^*(1)\smile\cdots p_{i-1}^*(1)\smile p_i^*(\ber)\smile p_{i+1}^*(1)\smile\cdots \smile p_n^*(-\ber)\\
&=p_i^*(\ber)-p_n^*(\ber)=\cer_i.
\end{aligned}
\]
\end{proof}
\forgotten

\m Similarly to \eqref{e:smile}, we put $\ber^2=\ber\smile\ber$,\ $\ber_i^2=\ber_i\smile\ber_i$, etc.

\begin{lemma}\label{l:bb}
 If $\ber^2\neq 0$ in $H^2(X;I(\pi)^{\otimes 2})$ then the cup product 
\[
\ber_1\smile \ber_2\smile\cdots \smile \ber_{n-2}\smile(\ber_{n-1})^2\in H^{n}( X^n ;I(p)^{\otimes n})
\]
is non-zero. 
\end{lemma}
   
\begin{proof}
By the definition of tensor product,
\[
0\neq \ber_1\otimes \ber_2\otimes\cdots \otimes \ber_{n-2}\otimes \ber_{n-1}\otimes \ber_{n-1}\in (H^1( X ;I(p))^{\otimes n}.
\]
Hence, because of monomorphicity of $\mu$ and by \propref{p:smile},
\[
\begin{aligned} 
0&\neq\mu(\ber_1\otimes \ber_2\otimes\cdots \otimes \ber_{n-2}\otimes \ber_{n-1}\otimes \ber_{n-1}\\
&=p_1^*(\ber)\smile \cdots \smile p_i^*(\ber)\smile\cdots \smile p_{n-2}^*(\ber)\smile p_{n-1}^*(\ber)\smile p_{n-1}^*(\ber)\\
&=\ber_1\smile \cdots \smile \ber_i\smile\cdots \smile \ber_{n-2}\smile \ber_{n-1}\smile \ber_{n-1}.
\end{aligned}
\]
\end{proof}

\m Take a point $*\in X$ and define
\[
\gf:X^{n}\to  X^n,\, (*,x_{n-1}, \ldots, x_1)\mapsto (x_n, \ldots, x_1).
\]
Consider the induced homomorphism 
\[
\gf^*:H^*(X^n;I(\pi)^{\ot n})\to H^*(X^n;I(\pi)^{\ot n}).
\]
Put $\ar_i=\ber_i-\ber_n$ for $i=1,\ldots, n-1$. Note that all $\ber_i$'s belong to the same group $H^*(X^n;I(\pi))$, and so the difference $\ber_i-\ber_n$ is well-defined.

\begin{lem}\label{l:phi}
$\gf^*(\ber_n)=0$ and $\gf^*(\ar_i)=\ber_i$ for $i<n$.
\end{lem}

\begin{proof}
\m Prove that $\gf^*(\ber_n)=0$. Indeed, $p_n\gf$ maps $X^n$ to the point, and so $\gf^*(\ber_n)=\gf^*p_n^*(\ber)=0$.

\m 
Prove that $\gf^*(\ber_i)=\ber_i$ for $i<n$. Indeed, $p_i\gf=p_i$ for $i<n$, and so 
\[
\ber_i=p_i^*(\ber)=\gf^*p_i^*(\ber)=\gf^*(\ber_i).
\]
Finally, $\gf^*(\ar_i)=\gf^*(\ber_i-\ber_n)=\gf^*(\ber_i)=\ber_i$ because $\gf^*(\ber_ n)=0$.
\end{proof}

\begin{lemma}\label{l:big} 
If $\ber^2\neq 0$ in $H^2(X;I(\pi)^{\otimes 2})$ then
\begin{equation}\label{e:big}
\left(\prod_{i=1}^{n-1}\ar_i\right)\smile\ar_{n-1}\neq 0
\end{equation}
in $H^{n}(X^{n};I(\pi)^{\otimes n})$.
\end{lemma}

\m Here $\prod_{i=1}^{n-1} v_i$ denotes $v_1\smile \cdots \smile v_{n-1}$ where $v_i\in H^n(X^{n};I(\pi))$.

\begin{proof} By \lemref{l:bb} we have
\[
\left(\prod_{i=1}^{n-1}\ber_i\right)\smile\ber_{n-1}=\ber_1\smile \ber_2\smile\cdots \smile \ber_{n-2}\smile(\ber_{n-1})^{\smile 2}\neq 0
\]
in $H^n( X^n ;I(p)^{\otimes n})$.

\m By \lemref{l:phi}, $\gf^*(\ber_n)=0$ and $\gf^*(\ar_i)=\ber_i$ for $i<n$. Hence
\[
\gf^*\left(\left(\prod_{i=1}^{n-1}\ar_i\right)\smile\ar_{n-1}\right)=\left(\prod_{i=1}^{n-1}\ber_i\right)\smile\ber_{n-1}\neq 0
\]
by \lemref{l:bb}. Thus 
\[
\left(\prod_{i=1}^{n-1}\ar_i\right)\smile\ar_{n-1}\neq 0.
\]
\end{proof}

\begin{cory}\label{c:n-1}
If $\ber^2\neq 0$ then $\TC_n(X)\geq n$.
\end{cory}

\begin{proof} Let $d:X\to X^n$ be the diagonal. Clearly $d^*(\ber_i)=\ber$ for all $i$, and so $d^*(\ar)=0$ for all $i$. The non-zero product 
\[
\left(\prod_{i=1}^{n-1}\ar_i\right)\smile\ar_{n-1}
\]
in \lemref{l:big} contains $n$ non-zero factors each of which is annihilated by the diagonal $d^*$. Hence $\zcl_n(X)\geq n$. Now the result follows from \theoref{t:zcn}. 
\end{proof}

\begin{prop}\label{p:n-1}
If $\TC_n(X)=n-1$ then $\cd(\pi)\leq 1$.
\end{prop}

\begin{proof}

Because of \corref{c:n-1} we have the following implication: If $\TC_n(X)=n-1$ then  $\ber^2=0$ in $H^2(X;I(\pi)^{\otimes 2})$. We can assume that the classifying space $B\pi$ can be obtained from $X$ by adding cells of dimensional $\geq 3$. Hence, 
$H^2(B\pi;I(\pi)^{\otimes 2})=0$, and so $\ber_{\pi}^2=0$, and the result follows from  \theoref{t:power}.
\end{proof}

\begin{thm}\label{main}
If $\TC_n(X)=n-1$ then either $X$ is homotopy equivalent to an odd-dimensional sphere $S^{2n+1}, n\geq 1$ or an integral homology circle with $\pi_1(X)=\Z$. 
\end{thm}

\begin{proof}
By \propref{p:n-1} we know that either $\cd(\pi)=0$ or $\cd (\pi)=1$. If $\cd(\pi)=0$ then $\pi$ is trivial group, and so $X$ is simply connected. In this case $X$ is homotopy equivalent to an odd-dimensional sphere $S^{2n+1}, n\geq 1$ by~\cite[Theorem 1.3]{GLO}.  

If $\cd(\pi)=1$ then $\pi$ is a nontrivial free group by~\cite{Stal, Swan}. Because of~\cite[Theorem 3.4]{GLO} we know that is either $X$ is acyclic, or $X$ is an integral homology sphere. First, $H_1(X)$ is a free nontrivial abelian group and hence $X$ cannot be acyclic. Furthermore, since $H_1(X)$ is a free abelian group and $X$ is an integral homology sphere, we conclude that $X$ is an integral homology circle. Finally, $\pi=\pi_1(X)=\Z$ since $\pi$ is a free group. 
\end{proof}

{\bf Acknowledgment:} I am grateful to Mark Grant for valuable discussions.

\end{document}